\documentclass[
11pt,%
tightenlines,%
twoside,%
onecolumn,%
nofloats,%
nobibnotes,%
nofootinbib,%
superscriptaddress,%
noshowpacs,%
centertags]%
{revtex4-2}
\usepackage{ljm}

\theoremstyle{plain}
\newtheorem*{maintheorem}{Main Theorem}

\newcommand{\RR}{\mathbb{R}} 
\newcommand{\CC}{\mathbb{C}} 
\newcommand{\NN}{\mathbb{N}} 
\newcommand{\ZZ}{\mathbb{Z}} 
\newcommand{\DD}{\mathbb{D}} 
\newcommand{\rad}{\mathfrak{r}}

\DeclareMathOperator{\Real}{Re}
\DeclareMathOperator{\dd}{d\!}
\DeclareMathOperator{\type}{type}

\DeclareMathOperator{\Zero}{\mathsf{Zero}}
\DeclareMathOperator{\Pol}{\mathsf{Pol}}

\usepackage[pdftex, unicode,colorlinks,
linkcolor=blue,citecolor=red,bookmarksopen,pdfhighlight=/N]{hyperref}

\begin{document}

\titlerunning{Decrease in growth of entire and meromorphic functions} 
\authorrunning{B. N. Khabibullin} 

\title{Decrease in growth of entire and  meromorphic functions}

\author{\firstname{B.~N.}~\surname{Khabibullin}}
\email[E-mail: ]{khabib-bulat@mail.ru}
\affiliation{Institute of Mathematics with Computer Center of the Ufa Science Center of the Russian Academy of Sciences, 112, Chernyshevsky str., Ufa,  Republic of Bashkortostan, Russian Federation,  450008}
\affiliation{Bashkir State Pedagogical University named after M. Akmulla, 
 3-a, Oktyabrskaya revolyutsiya str., Ufa,  Republic of Bashkortostan, Russian Federation, 450077}

\received{Marz 5, 2025}

\begin{abstract} 
We solve  the following three problems. 

1. How much can the radial growth of an entire function $f$ be reduced by multiplying it by some nonzero entire function? We give the answer in terms of the growth of the integral means of $\ln|f|$ over the circles centered at the origin.

2. We estimate the smallest possible radial growth of non zero entire  functions that vanish on a given distribution of points $Z$.  We solve this problem in terms of the growth of the  radial  integral counting function of $Z$.

3. Let $F=f/g$ be a meromorphic function with representations as the ratio of  entire functions $f\neq 0$ and $g\neq 0$. How small can the radial growth of entire functions $f$ and $g$ be in such representations in relation to the growth of the Nevanlinna characteristic of  $F$?

All solutions have a non-asymptotic uniform character, and the obtained inequalities are sharp. 
All of them are based on some main theorem for subharmonic functions, which relies on the Govorov\,--\,Petrenko\,--\,Dahlberg\,--\,Ess\'en inequality and uses our general results on the existence of subharmonic minorants.
\end{abstract}



\maketitle


\section{Main results}

We denote by   $\NN:=\{1,2,\dots\}$,    $\ZZ$, $\RR$, $\CC$  the sets of \textit{natural, integer, real, complex\/} numbers  respectively. Besides, $\NN_0:=\{0\}\cup \NN$, 
$\overline\NN_0:=\NN_0\cup \{+\infty\}$, $\RR_+:=\bigl\{r\in \RR\bigm| r\geq 0\bigr\}$, 
 $\overline{\RR}:=\RR\cup\{\pm\infty\}$, 
$\overline{\RR}_+:=\RR_+\cup\{+\infty\}$.
All these sets, whenever possible, are provided with a standard order relation and  geometric, algebraic,  topological structures on them.
The \textit{positive part\/} of  a quantity  $a\in\overline{\RR}$ or function $a\colon X\to \overline{\RR}$ is denoted by $a^+:=\sup\{a,0\}$.

 \subsection{Main results for entire 
functions}
Let $f$ be an entire function on the complex plane $\CC$.  We define  \cite{GO}  
\begin{align}
M(r;f)&\underset{r\in \RR_+}{:=} \max_{|z|=r} \bigl|f(z)\bigr|,\label{dfMf}\\
 C(r;f)&\underset{r>0}{:=}\frac{1}{2\pi}\int_{0}^{2\pi} \ln \bigl|f(re^{i\theta})\bigr|\dd \theta, 
\quad C(0,f):=\ln\bigl|f(0)\bigr|.
\label{dfCf}
\end{align}

Our first result is the following multiplier theorem for entire functions. 
 
\begin{theorem}
\label{Cor1} 
For each entire function $f$ with $f(0)\neq 0$ and for every number $p\geq 1$,  there exists an entire function $h\neq 0$ such that 
\begin{equation}\label{eshe}
 M(r;fh)\leq \exp \int_0^{+\infty} C( rt; f)\frac{p^2t^{p-1}\dd t}{(1+t^p)^2} \quad\text{for all $r\in \RR_+$},
\end{equation}
but there is no entire function $h\neq 0$ for which we have the relation
\begin{equation}\label{M/C}
    \liminf_{r\to +\infty}\frac{M(r;fh)}{\exp C(r;f)}=0.
\end{equation}
\end{theorem}
The second part of Theorem \ref{Cor1} with relation \eqref{M/C} is rather trivial, while the first part with inequality \eqref{eshe} is quite profound.
The following Corollary \ref{Coruniq}  significantly develops our results from \cite{Kha92}.

\begin{corollary}\label{Coruniq}
Let $0<\sigma\in \RR$, $0<\rho\in \RR$. If  $f$ is an entire function such that $f(0)\neq 0$ and 
\begin{equation}\label{eqCrs}
    \sup_{r\in \RR_+}\bigl(C(r;f)-\sigma r^{\rho}\bigr)<+\infty,
\end{equation}
then there exists an entire function $h\neq 0$ such that 
\begin{equation}\label{Mh}
    M(r;fh)\leq e^{P(\rho)\sigma r^{\rho}}\quad \text{for all $r\in \RR_+$}
\end{equation}
where, for any $\rho>0$,  it is impossible to decrease the Paley constant 
\begin{equation}\label{Paley}
    P(\rho):=\begin{cases}
    \pi \rho &  \text{if\/ $\rho\geq \frac12$},\\
    \dfrac{\pi\rho}{\sin \pi\rho}  &  \text{if\/ $0<\rho<\frac12$.}
\end{cases}
\end{equation} 
Therefore, in particular, if 
$$
 \limsup_{r\to +\infty}\frac{C(r;f)}{r^{\rho}}<\sigma,
$$
then there is an entire function $h\neq 0$ such that 
\begin{equation*}
  \limsup_{z\to \infty}\frac{\ln\bigl|f(z)h(z)\bigr|}{|z|^{\rho}}<P(\rho)\sigma.
\end{equation*}
\end{corollary}

\subsection{Main results on distribution of roots of entire functions}
We will call an arbitrary function ${\mathsf{Z}}\colon \CC\to \overline{\NN}_0$  the \textit{distribution of points on\/} $\CC$, and for each point~$z\in \CC$, the value ${\mathsf{Z}}(z)\in \overline{\NN}_0$ is the \textit{multiplicity of ${\mathsf{Z}}$ at $z$.}
 For every distribution of points ${\mathsf{Z}}$ on $\CC$, we define its \textit{radial counting function}
\begin{equation*}
    {\mathsf{Z}}^{\rad} \colon r \underset{r\in \RR_+}{\longmapsto}\sum_{|z|\leq r} {\mathsf{Z}}(z)\in \overline{\RR}_+ 
\end{equation*}
as well as its  \textit{radial integral counting function}
\begin{equation}\label{NZ}
N_{\mathsf{Z}}\colon r \underset{r\in \RR_+}{\longmapsto}\int_{0}^{r}\frac{{\mathsf{Z}}^{\rad}(t)-{\mathsf{Z}}^{\rad}(0)}{t}\dd t
+\mathsf{Z}^{\rad}(0)\ln r\in \overline{\RR}_+.
\end{equation}
Let  $f$ be an entire function. The distribution of points 
\begin{equation*}
	{\Zero}_f\colon  z\underset{z\in \mathbb{C}}{\longmapsto}\sup\biggl\{n\in  \overline{\mathbb{N}}_0\biggm| \limsup_{z\neq w\to z}\frac{|f(w)|}{|w-z|^n}<+\infty \biggr\}
\end{equation*}
will be called the \textit{distribution of roots\/} for this entire  function $f$. 
Thus, $	{\Zero}_f(z)$ is the \textit{multiplicity of root\/} of the function $f$ at the point $z\in \CC$. 
If $f=0$, then $\Zero_0=+\infty$, and  $f\neq 0$ if and only if 
${\Zero}_f^{\rad}(r)<+\infty$ for each $r\in \RR_+$, i.e., the radial counting function ${\Zero}_f^{\rad}$   is a finite function. The entire function $f$ \textit{vanishes on the distribution of points\/} $Z$ and \textit{we write\/} $f({\mathsf{Z}})=0$ if ${\Zero}_f\geq {\mathsf{Z}}$\/ on $\CC$. 

The following theorem estimates the smallest radial growth of entire functions vanishing on ${\mathsf{Z}}$ in terms of the radial integral counting function $N_{\mathsf{Z}}$.

\begin{theorem}\label{Th2}  Let  ${\mathsf{Z}}$ be a distribution of points on $\CC$
 such that  ${\mathsf{Z}}^{\rad}(r)<+\infty$ for each $r\in \RR_+$ and  ${\mathsf{Z}}(0)=0$. 
Then for each number $p\geq 1$ there is an entire function $f\neq 0$ such that 
$f(\mathsf{Z})=0$ and 
\begin{equation}\label{esheZ}
    M(r;f)\leq \exp 
 \int_0^{+\infty}
N_{\mathsf{Z}}( rt)\frac{p^2t^{p-1}\dd t}{(1+t^p)^2}\quad\text{for all $z\in \CC$},
\end{equation}
but there is no entire function $f\neq 0$ such that $f({\mathsf{Z}})=0$ and 
\begin{equation}\label{M/CZ}
    \liminf_{r\to +\infty}\frac{M(r,f)}{\exp N_{\mathsf{Z}}(r)}=0.
\end{equation}
\end{theorem}
The second part of Theorem \ref{Th2} with relation \eqref{M/CZ} is rather trivial, while the first part with~\eqref{esheZ} is quite profound. The following Corollary \ref{CoruniqZ} contains our asymptotic results from \cite[Theorem 1]{Kha92}, significantly strengthens them and shows that  relations \eqref{esheZ}--\eqref{M/CZ} of Theorem~\ref{Cor1} are sharp.
\begin{corollary}\label{CoruniqZ}
Let $0<\sigma \in \RR$,  $0<\rho \in \RR$. If\/  ${\mathsf{Z}}$ is  a distribution of points such that  
\begin{equation*}
    \sup_{r\in \RR_+}\bigl(N_{\mathsf{Z}}(r)-\sigma r^{\rho}\bigr)<+\infty,
\end{equation*}
then there exists an entire function $f\neq 0$ such that $f(\mathsf{Z})=0$ and 
\begin{equation*}
    M(r,f)\leq e^{P(\rho)\sigma r^{\rho}}\quad \text{for all $r\in \RR_+$}
\end{equation*}
where, for any $\rho>0$,  it is impossible to decrease the Paley constant $P(\rho)$ defined in \eqref{Paley}. 

Therefore, in particular, if 
$$
 \limsup_{r\to +\infty}\frac{N_{\mathsf{Z}}(r)}{r^{\rho}}<\sigma,
$$
then there is an entire function $f\neq 0$ such that $f(\mathsf{Z})=0$ and 
\begin{equation*}
  \limsup_{z\to \infty}\frac{\ln\bigl|f(z)\bigr|}{|z|^{\rho}}<P(\rho)\sigma.
\end{equation*}

\end{corollary}

\subsection{Main results for meromorphic functions}

Let $F=\dfrac{f}{g}$ be a meromorphic function on $\CC$ represented as the ratio of  entire functions $f$ and $g$ such that  $f\neq 0$ or $g\neq 0$. 
Then  the \textit{distribution of poles ${\Pol}_F$ for\/} $F$ can be defined as 
\begin{equation}\label{dfPof}
    {\Pol}_F:=(\Zero_g-\Zero_f)^+\colon z\underset{z\in \CC}{\longmapsto} \bigl(\Zero_g(z)-\Zero_f(z)\bigr)^+\in \overline{\NN}_0.
\end{equation}
The classical Nevanlinna characteristic of $F$ is the function \cite{GO}
\begin{equation}\label{dfTf}
    T(r;F)\underset{r\in \RR_+}{:=} \frac{1}{2\pi}\int_0^{2\pi}\ln^+\bigl|F(re^{i\theta})\bigr|\dd \theta
+N_{\Pol_F}(r).
\end{equation}

\begin{theorem}\label{Th3} For any number $p\geq 1$, each meromorphic function $F$ on $\CC$ with $F(0)\neq 0,\infty$ can be represented as the ratio $F=\dfrac{f}{g}$ of entire  functions $f$ and $g$ such that 
\begin{equation}\label{eshem}
    \max\bigl\{M(r;f),M(r;g)\bigr\}\leq \exp  \int_0^{+\infty}
T( rt;F)\frac{p^2t^{p-1}\dd t}{(1+t^p)^2}\quad\text{for all $z\in \CC$},
\end{equation}
but there is no a pair of entire functions $f,g$ such that $F=\dfrac{f}{g}$  and 
\begin{equation}\label{M/Cm}
    \liminf_{r\to +\infty}\frac{\max\bigl\{M(r;f),M(r;g)\bigr\}}{\exp T(r,F)}=0.
\end{equation}
\end{theorem}

The following Corollary \ref{CoruniqF} contains our asymptotic results from \cite[Theorem 2]{Kha92}. 
\begin{corollary}\label{CoruniqF}
Let $0<\sigma, \rho  \in \RR$. If   $F$ is a meromorphic function on $\CC$ with $F(0)\neq 0,\infty$ such that  
\begin{equation*}
    \sup_{r\geq 0}\bigl(T(r;F)-\sigma r^{\rho}\bigr)<+\infty,
\end{equation*}
then there exists a pair of entire functions  $f\neq 0$ and $g\neq 0$ such that $F=\dfrac{f}{g}$ and 
\begin{equation*}
\max\bigl\{M(r;f),M(r;g)\bigr\}\leq  e^{P(\rho)\sigma r^{\rho}}\quad \text{for all $r\in \RR_+$}
\end{equation*}
where, for any $\rho>0$,  it is impossible to decrease the Paley constant $P(\rho)$ defined in \eqref{Paley}. 

Therefore, in particular, if 
$$
 \limsup_{r\to +\infty}\frac{T(r;F)}{r^{\rho}}<\sigma,
$$
then there is a pair of entire functions  $f\neq 0$ and $g\neq 0$ representing  $F=\dfrac{f}{g}$ such that 
\begin{equation*}
  \limsup_{z\to \infty}\frac{\ln\bigl|f(z)\bigr|}{|z|^{\rho}}<P(\rho)\sigma 
\quad\text{and}\quad \limsup_{z\to \infty}\frac{\ln\bigl|g(z)\bigr|}{|z|^{\rho}}<P(\rho)\sigma.
\end{equation*}
\end{corollary}

\subsection{Main result  for subharmonic functions}

We denote by 
\begin{equation*}
    \DD:=\bigl\{z\in \CC\bigm| |z|<1\bigr\},\quad   \overline{\DD}:=\bigl\{z\in \CC\bigm| |z|\leq 1\bigr\},  \quad \partial  \DD=\partial  \overline{\DD}:=\overline{\DD}\setminus \DD 
\end{equation*} 
the \textit{open unit disk,\/} the \textit{closed unit disk,\/} the \textit{unit circle,\/} respectively. 

For $d\in \RR_+$, we denote by $\mathfrak{m}_d$ the $d$-dimensional Hausdorff measure in $\CC$ \cite{EG},  \cite[2.4]{Kha22AA}, \cite[\S~5]{Kha22MS}. 
Thus,  $\mathfrak{m}_2$ is the planar Lebesgue measure \cite{EG}, the restriction of $\mathfrak{m}_1$ on a Lipschitz arc or closed curve is a measure of its length \cite{EG}, \cite[\S~5]{Kha22MS},  and $\mathfrak{m}_0(S)=\sum\limits_{s\in S}1$ is equal to the number of points in $S$. 

Let  $u\neq -\infty$ be a subharmonic function on  $\CC$. For $r\in \RR_+$, we define \cite[Definition 2.6.7]{Rans},  \cite{Kha22AA},  \cite{Kha22MS}   
\begin{subequations}\label{chu}
\begin{align}
   {\mathsf{M}}_u(r)&\underset{r>0}{:=}\sup_{r\overline{\DD}}u=\sup_{r\partial\overline{\DD}}u, &{\mathsf{M}}_u(0):=u(0),
\tag{\ref{chu}M}\label{uM}
\\
\intertext{the \textit{integral mean of $u$ over the circle\/} $r\partial \overline{\DD}$ of radius $r$ centered at $0$}
\mathsf{C}_u(r)&\underset{r>0}{:=}\frac{1}{2\pi r}\int_{r\partial\overline{\DD}}u\dd\mathfrak{m}_1=\frac{1}{2\pi}\int_{0}^{2\pi} u(re^{i\theta})\dd \theta, &{\mathsf{C}}_u(0):=u(0),
\tag{\ref{chu}C}\label{uC}
\\
\intertext{the \textit{integral mean of $u$ over the disk\/} $r\overline{\DD}$ of radius $r$ centered at $0$}
\mathsf{B}_u(r)&\underset{r>0}{:=}\frac{1}{\pi r^2}\int_{r\overline{\DD}}u\dd \mathfrak{m}_2=
 \frac{2}{r^2}\int_{0}^{r} \mathsf{C}_u(t)t\dd t, &{\mathsf{B}}_u(0):=u(0),
\tag{\ref{chu}B}\label{uB}
\\
\intertext{and the \textit{\it Nevanlinna  characteristic\/}} 
\mathsf{T}_u(r)&:=\mathsf{C}_{u^+}=\frac{1}{2\pi}\int_{0}^{2\pi}u^+(re^{i\theta})\dd \theta.
\tag{\ref{chu}T}\label{uT}
\end{align}
\end{subequations}

For each subharmonic function on $\CC$ we have inequalities \cite[Theorem 2.6.8]{Rans}, \cite{Kha22AA},  \cite{Kha22MS}   
\begin{equation}\label{0BC}
    u(0)\underset{r\in \RR_+}{\leq} \mathsf{B}_u(r)\underset{r\in \RR_+}{\leq} \mathsf{C}_u(r)\underset{r\in \RR_+}{\leq}
\begin{cases}
 {\mathsf{M}}_{u}(r)\underset{0<r\to 0}{\longrightarrow} u(0).  \\
 \mathsf{C}_{u}^+(r)\underset{r\in \RR_+}{\leq} \mathsf{T}_u(r)
\underset{r\in \RR_+}{\leq}  \mathsf{M}_u^+(r).
\end{cases}
\end{equation}
All these functions-characteristics in \eqref{0BC} are increasing continuous convex of $\ln$, 
i.e., their superposition with the exponential function $\exp$ are convex on $\RR$, and   their superposition with  the modulus  $z\underset{z\in \CC}{\longmapsto}|z|$ are subharmonic functions  on $\CC$.
If $f$ is an entire function, then $\ln|f|$ is subharmonic  and, by definitions \eqref{dfMf}--\eqref{dfCf} and 
\eqref{dfPof}--\eqref{dfTf}, 
\begin{equation}\label{MCTu}
    M(r;f)\overset{\eqref{dfMf},\eqref{uM}}{\underset{r\in \RR_+}{=}}\exp \mathsf{M}_{\ln|f|}(r),\qquad 
C(r;f)\overset{\eqref{dfCf},\eqref{uC}}{\underset{r\in \RR_+}{=}}\mathsf{C}_{\ln|f|}(r), \qquad
T(r;f)
\overset{\eqref{dfTf},\eqref{uT}}{\underset{r\in \RR_+}{=}}
\mathsf{T}_{\ln|f|}(r).
\end{equation}

Our  main result is the following 
\begin{maintheorem}\label{Th1} Let $u$ be a subharmonic function on $\CC$ that is harmonic in some neighborhood of the origin. Then for any number  $p\geq 1$ there is an entire function $h\neq 0$ such that
\begin{equation}\label{esh0}
    u(z)+\ln|h(z)|\leq p^2 \int_0^{+\infty}
\mathsf{C}_u\bigl(|z|t\bigr)\frac{t^{p-1}\dd t}{(1+t^p)^2}
\quad\text{for all $z\in \CC$,}
\end{equation}
where the integral on the right side is finite for all $z\in \CC$ if and only if
\begin{equation}\label{iT}
    \int_1^{+\infty}\frac{\mathsf{C}_u(t)}{t^{p+1}}\dd t< +\infty.
\end{equation}
\end{maintheorem} 
  


\section{Preliminary results}

\subsection{Representations of subharmonic functions}

\begin{lemma}\label{Lemr} Let $u\neq -\infty$ be a subharmonic function on $\CC$ satisfying \eqref{iT} for some number $p>0$. Then there are 
a subharmonic  function $U$ on $\CC$ of zero type over the order $p$, i.e., 
\begin{equation}\label{M0U}
\type_p [u]:= \limsup_{z\to \infty}\frac{\mathsf{M}_u(|z|)}{|z|^p}=0, 
\end{equation}
and an entire function $F$  that does not vanish, such that $u=U+\ln|F|$.
\end{lemma}
\begin{proof}
It is follows from  \eqref{iT} that 
\begin{equation}\label{C0}
    \mathsf{C}_u(r)\underset{r>1}{\leq}pr^p\int_r^{+\infty}\frac{\mathsf{C}_u(t)}{t^{p+1}}\dd t=o(r^p) \quad\text{as $r\to +\infty$.}
\end{equation}
The {\it  Riesz distribution of masses\/} for  $u$ is a positive Borel measure 
\begin{equation*}
\varDelta_u:= \frac{1}{2\pi} {\bigtriangleup}{u},
\end{equation*}
where $\bigtriangleup$ is  the {\it Laplace operator\/}  acting in the sense of the
theory of distributions or generalized functions. We denote by 
\begin{equation*}
    \varDelta_u^{\rad}\colon t\underset{t\in \RR_+}{\longmapsto} \varDelta_u\bigl(t\overline{\DD}\bigr)
\end{equation*}
the \textit{radial counting function of\/} $\varDelta_u$.
By  the Poisson\,--\,Jensen\,--\,Nevanlinna\,--\,Privalov formula for subharmonic functions
\cite[4.5]{Rans}, \cite[Ch.~2,\S~2]{Priv37}, \cite[\S~3,  (3.2)]{Kha22AA} we have
\begin{equation*}
   N_{\varDelta_u}(r,R):=\int_r^R\frac{ \varDelta_u^{\rad}(t)}{t}\dd t
= \mathsf{C}_u(R)-\mathsf{C}_u(r) \quad\text{for all $0<r<R<+\infty$.}
\end{equation*}
Hence, by \eqref{iT} and \eqref{C0}, we obtain
\begin{align}
    \int_1^{+\infty}\frac{N_{\varDelta_u}(1,t)}{t^{p+1}}\dd t<+\infty,
\label{Ni}
\\
\limsup_{r\to +\infty}\frac{N_{\varDelta_u}(1,r)}{r^p}=0.
\label{Np}
\end{align}
We denote the \textit{integer part\/} of the number $p$ as 
\begin{equation*}
    \lfloor p \rfloor:=:
\text{\sf floor} (p):=  \sup\{k\in {\mathbb{Z}}\colon  k\leq p\} \in \mathbb{Z},\quad \lfloor p \rfloor\leq p ,
\label{{fc}f}
\end{equation*}
and the \textit{smallest integer greater than}  $p$ as
\begin{equation*}
\lceil p \rceil :=:\text{\sf ceil} (p):=  \inf\{k\in {\mathbb{Z}} \colon  k\geq p\} 
=\begin{cases}
    p  &  \text{if $p\in \ZZ$} \\
     \lfloor p \rfloor+1 &  \text{if $p\notin \ZZ$} 
\end{cases}\quad \in \mathbb{Z}, \quad \lceil p \rceil\geq p\geq  \lfloor p \rfloor .
\label{{fc}c}
\end{equation*}

By the Weierstrass\,--\,Hadamard\,--\,Brelot Theorem on representation of subharmonic function \cite[4.2]{HK} the following function
\begin{equation*}
    U\colon z\underset{z\in \CC}{\longmapsto}
\int_{\DD}\ln|w-z|\dd \varDelta_u(w)+\int_{\CC\setminus \DD}
\underset{K_{\lceil p \rceil -1}(z,w)}{\underbrace{\biggl(\ln\Bigl|1-\frac{z}{w}\Bigr|+ \sum_{j=0}^{\lceil p \rceil -1}\Real\frac{z^j}{w^j}\biggr)}}\dd\varDelta_u(w)
\end{equation*}
defines a possible subharmonic function with Riesz distribution of masses $\varDelta_u$ and $\type_p[U]=0$, that   
follows from the following upper  estimate of second integral  \cite[Lemma 4.4, (4.2.4)]{HK}: 
\begin{equation*}
\int_{\CC\setminus \overline{\DD}}
K_{\lceil p \rceil -1}(z,w)
\dd\varDelta_u(w)\underset{r:=|z|\geq 1}{\leq} C_p\biggl( r^{\lceil p \rceil -1}\int_1^r\frac{N_{\varDelta_u}(1,t)}{t^{\lceil p \rceil}}\dd t 
+r^{\lceil p \rceil}\int_r^{+\infty}\frac{N_{\varDelta_u}(1,t)}{t^{\lceil p \rceil+1}}\dd t \biggr),
\end{equation*}
where the constant $C_p\in\RR_+$ depends only on $p>0$, and 
\begin{gather*}
    r^{\lceil p \rceil -1}\int_1^r\frac{N_{\varDelta_u}(1,t)}{t^{\lceil p \rceil}}\dd t 
\underset{r,t\to +\infty}{\overset{\eqref{Np}}{=}}r^{\lceil p \rceil -1}
\int_1^r\frac{o(t^p)}{t^{\lceil p \rceil }}\dd t
\underset{r\to +\infty}{=}r^{\lceil p \rceil -1}\cdot o\bigl(r^{p+1-\lceil p \rceil}\bigr)
\underset{r\to +\infty}{=}o(r^p), \\
  r^{\lceil p \rceil }\int_r^{+\infty}\frac{N_{\varDelta_u}(1,t)}{t^{\lceil p \rceil+1}}\dd t 
=
  r^{\lceil p \rceil }\int_r^{+\infty}\frac{1}{t^{\lceil p \rceil -p}}\cdot \frac{N_{\varDelta_u}(1,t)}{t^{p+1}}\dd t 
\underset{r\geq 1}{\leq} r^p
\int_r^{+\infty}\frac{N_{\varDelta_u}(1,t)}{t^{p+1}}\dd t 
\overset{\eqref{Ni}}{\underset{r\to +\infty}{=}}o(r^p). 
\end{gather*}
Besides, $\varDelta_U=\varDelta_u$, and, by Weyl's lemma \cite[Lemma 3.7.10]{Rans}, there is a harmonic function $H$
on $\CC$ such that  $u=U+H$. 
By \cite[Theorem 1.1.2(b)]{Rans} there is an entire function $F\neq 0$ without zeros such that $H=\ln|F|$ and 
$u=U+\ln |F|$. 
\end{proof}

\subsection{Variation of the Govorov\,--\,Petrenko\,--\,Dahlberg\,--\,Ess\'en inequality}

In \cite{Govorov}, 1964,  N.V. Govorov used one very useful sharp inequality for entire functions to solve the classical Paley problem of 1932 (see also \cite[Appendix]{GO}, \cite{Govorovbook}). In \cite{Petrenko}, 1969, V. P. Petrenko extended this inequality to meromorphic functions. In \cite{Dahlberg}, 1972, B. Dahlberg extended this inequality to subharmonic functions in the space of three or more variables. In \cite{Essen}, 1975, M. Ess\'en adapted the latter inequality for subharmonic functions on the complex plane.
We will formulate this inequality as close as possible to the original formulations in \cite{Dahlberg} and \cite{Essen}.

\begin{lemma}[{\cite[Lemma 3.4]{Dahlberg}, \cite[Lemma 4.1]{Essen}}]\label{LemEssen} Let $u\geq 0$ be a two times continuously differentiable subharmonic function on $\CC$. 
Suppose that the function $u$ is harmonic in some neighborhood of the origin, and $\frac12<\lambda\in \RR$. 
Then there is a number $C>0$, only depending on $\lambda$, such that 
\begin{equation*}
    u(r)\leq V(u,r,R)+C\mathsf{M}_u(6R)(r/R)^{2\lambda}\quad\text{for all $0<r<R/2$,}
\end{equation*} 
where 
\begin{align*}
 V(u,r,R)&:=-(2\pi)^{-1}\iint_{K_R} u(z)\Delta \Psi(r,z)\dd \mathfrak{m}_2(z)\\
K_R&:=\Bigl\{\rho e^{i\theta}\in \CC\Bigm| 0<\rho<R, -\frac{\pi}{2\lambda}<\theta<\frac{\pi}{2\lambda}\Bigr\}\subset \CC\\
\Psi(r,w)&:=\Psi\bigl(r,|w|\bigr)=\Psi(r,\rho):=-\ln \biggl(1+\Bigl(\frac{\rho}{r}\Bigr)^{2\lambda}\biggr) \text{ for $w:=\rho e^{i\theta}$},\\
-\rho {\bigtriangleup} \Psi(r,w)&=4\lambda^2\rho^{-1}\bigl((r/\rho)^{\lambda}+(\rho/r)^{\lambda}\bigr)^{-2}.
\end{align*}
\end{lemma}
Direct explicit calculations of functions and integrals in polar coordinates allow us to obtain the following corollary of Lemma \ref{LemEssen} for \textit{positive\/} subharmonic functions.

\begin{lemma}\label{Lem22} Let $u\geq 0$ be a subharmonic function on $\CC$ that is harmonic in some neighborhood of the origin, and $ \frac{1}{2}\leq\lambda\in \RR_+$.  
Then 
there is $C\in \RR_+$ such that 
\begin{equation*}
\mathsf{M}_u(r)\leq  4\lambda^2\int_0^{R}\frac{\mathsf{C}_u(t)}{t\bigl((r/t)^{\lambda}+(t/r)^{\lambda}\bigr)^2}\dd t
+C\mathsf{M}_u(6R)\Bigl(\frac{r}{R}\Bigr)^{2\lambda}\quad \text{for all $R\geq 2r>0$}. 
\end{equation*}
Hence, if  $u\geq 0$ is a subharmonic function of zero type over the order $p=2\lambda\geq 1$, i.e., 
$\type_p [u]\overset{\eqref{M0U}}{=}0$, 
  then, for  $r:=1$  and $R\to +\infty$,  we have  the inequality
\begin{equation}\label{MUr}
u(1)\leq  p\int_0^{+\infty}\mathsf{C}_u(t)\frac{\dd t^p}{\bigl(1+t^{p}\bigr)^2}
\overset{\eqref{uC}}{=}p\int_0^{+\infty}\biggl(\frac{1}{2\pi}\int_0^{2\pi}u(te^{i\theta})\dd\theta \biggr)\frac{\dd t^p}{\bigl(1+t^{p}\bigr)^2}. 
\end{equation}

\end{lemma}
We omit the detailed derivation of Lemma \ref{Lem22} from Lemma \ref{LemEssen}  for two reasons. Firstly, similar arguments of this kind were carried out in \cite[Calculations (4.10)--(4.15)]{Essen}, \cite[Lemmas 2, 3]{Kha92}, \cite[\S~1]{Kha94}. 
Secondly, we have proved much more general inequalities of this type based on the works \cite{Kha03}, \cite{Kha21}, and the main result in~\cite[Theorem 3]{KhaMen21}. These inequalities will be described elsewhere, and Lemma \ref{Lem22} will be deduced from them as a very special case.

\section{Proof of  the Main Theorem}

If $u=-\infty$ or the right side of \eqref{esh0} is equal to $+\infty$ for all $z\neq 0$, then any entire function $h$ with 
\begin{equation*}
 0\neq   h(0)\leq \mathsf{C}_u(0)\int_0^{+\infty}\frac{p^2t^{p-1}\dd t}{(1+t^{p})^2}=
\mathsf{C}_u(0)p=pu(0)
\end{equation*} 
is suitable. If the integral in \eqref{esh0} is finite for some  $z\in \CC\setminus \{0\}$, then there is  $c\in \RR_+\setminus \{0\}$ such that 
\begin{equation}\label{icT}
    \int_0^{+\infty}
\mathsf{C}_u^+\bigl( ct\bigr)\frac{t^{p-1}\dd t}{(1+t^p)^2}<+\infty. 
\end{equation}
The change of variable $ct$ by new variable in the integrand entails the finiteness of the integral \eqref{iT}. 
Conversely, if the integral \eqref{iT} is finite, then applying the inverse substitution gives the finiteness of the integral \eqref{icT} for each $c\in \RR_+\setminus \{0\}$.
Therefore, we can further assume  \eqref{iT} for $u\neq -\infty$. 

We use the representation $u=U+\ln|F|$ of  Lemma \ref{Lemr} and have 
\begin{gather*}
\mathsf{C}_{\ln|F|}(r)=\ln|F(0)|\neq -\infty\quad\text{for all $r\in \RR_+$, $F(0)\neq 0$}, 
\\
    p^2 \int_0^{+\infty}
\mathsf{C}_{\ln|F|}\bigl(|z|t\bigr)\frac{t^{p-1}\dd t}{(1+t^p)^2}
\underset{z\in \CC}{=}p^2 \int_0^{+\infty}
\frac{t^{p-1}\dd t}{(1+t^p)^2}\cdot \ln|F(0)|\underset{z\in \CC}{=}p\ln|F(0)|.
\end{gather*}
 Hence, for entire function  $G:=|F(0)|^p/F$ we obtain 
\begin{equation*}
    \ln|F|+\ln|G|=\ln \Bigl|F\cdot  \frac{|F(0)|^p}{F} \Bigr|=p\ln|F(0)|\underset{z\in \CC}{=}
p^2 \int_0^{+\infty}\mathsf{C}_{\ln|F|}\bigl(|z|t\bigr)\frac{t^{p-1}\dd t}{(1+t^p)^2}.
\end{equation*}
Thus, our Main Theorem is proven for special  subharmonic function $u:=\ln|F|$ where $F$ is an entire function without zeros, 
and, due to the additivity of inequality \eqref{esh0} with respect to the addition of functions $u$,  it suffices to prove the Main Theorem  for subharmonic functions 
\begin{equation}\label{typeu0}
    u=U\neq -\infty,\quad   \type_p[u]=\type_p[U]=0
\end{equation} 
under the condition \eqref{iT}. We  temporarily impose that 
\begin{equation}\label{0}
    u(0)=0. 
\end{equation}

\begin{definition}\label{defJ} А Borel measure $\mu\geq 0$ with compact  support\/ $\operatorname{supp} \mu\subset \CC$ is  a Jensen measure if 
\begin{equation}\label{defJm}
    u(0)\leq \int_{\CC}u\dd \mu\quad\text{for each subharmonic function $u$ on $\CC$}.
\end{equation}
\end{definition}
Let $\mu$ be a Jensen measure with $\operatorname{supp} \mu\subset R_{\mu}\overline{\DD}$, $R_\mu\in \RR_+$. 
Evidently, every Jensen measure $\mu$ is probability and  $\mu(\CC)=1$.
Consider the function
\begin{equation}\label{Udf}
    u^{\mu}\colon w\underset{w\in \CC}{\longmapsto} \int_{\CC}u(wz)\dd \mu(z)
= \int_{R_{\mu}\DD}u(wz)\dd \mu(z) 
\end{equation} 
where $u$ is our function satisfying \eqref{typeu0}--\eqref{0}. 
Such function $u^{\mu}$ is subharmonic  \cite[Theorem 2.4.8]{Rans} 
and harmonic  on some neighborhood of the origin under the conditions \eqref{typeu0}--\eqref{0} since 
$\operatorname{supp} \mu$ is compact  subset in $\CC$.
Besides, by Definition \ref{defJ} we obtain 
\begin{equation*}
 u^{\mu}(w)\overset{\eqref{Udf}}{=}\int_{\CC}u(wz)\dd \mu(z)\overset{\eqref{defJm}}{\geq} u(w\cdot 0)=u(0)\overset{\eqref{0}}{=}0\quad \text{for all $w\in \CC$.}
\end{equation*}
It follows from construction  \eqref{Udf}  that, for any $r\in \RR_+$, 
\begin{equation*}
 \mathsf{M}_{u^{\mu}}(r)\overset{\eqref{Udf}}{\leq} \sup_{|z|\leq r}
\int_{R_{\mu}\DD}u(wz)\dd \mu(z) \leq 
 \sup_{t\leq R_{\mu}}     \mathsf{M}_{u}(tr)\mu(\CC)\leq   \mathsf{M}_{u}(R_{\mu}r).
\end{equation*}
Hence, in view of \eqref{typeu0}, we have $\type_p[u^{\mu}]=0$. 
Therefore, by Lemma \ref{Lem22}  we obtain \eqref{MUr} for $u^{\mu}$ in the role $u$. 
Hence, if we return to the definition \eqref{Udf} of the function $u^{\mu}$, then we have 
\begin{equation*}
\int_{\CC}u\dd \mu \overset{\eqref{Udf}}{=}u^{\mu}(1)\overset{\eqref{MUr}}{\leq} 
p\int_0^{+\infty}\biggl(\frac{1}{2\pi}\int_0^{2\pi}
\int_{\CC}u(te^{i\theta}z)\dd \mu(z)
\dd\theta \biggr)\frac{\dd t^p}{\bigl(1+t^{p}\bigr)^2}.
\end{equation*}
By Fubini's theorem on repeated integrals, we get
\begin{align*}
\int_{\CC}u\dd \mu &\leq 
p\int_0^{+\infty}\frac{1}{2\pi}\int_0^{2\pi}
\int_{\CC}u(te^{i\theta}z)\dd \mu(z)
\dd\theta \frac{\dd t^p}{\bigl(1+t^{p}\bigr)^2}\\
&=
\int_{\CC}\underset{M(z)}{\underbrace{p\int_0^{+\infty}\Biggl(\frac{1}{2\pi}\int_0^{2\pi}
u(te^{i\theta}z)\dd\theta\ \Biggr)\frac{\dd t^p}{\bigl(1+t^{p}\bigr)^2}}}\dd \mu(z)
\\&=
\int_{\CC}
\underset{M(z)}{\underbrace{\Biggl(p\int_0^{+\infty}\mathsf{C}_u\bigl(t|z|\bigr)\frac{\dd t^p}{\bigl(1+t^{p}\bigr)^2}\Biggr)}}
\dd \mu(z),
\end{align*}
where the last integrand 
\begin{equation}\label{TI}
M\colon z\underset{z\in \CC}{\longmapsto}p\int_0^{+\infty}\mathsf{C}_u\bigl(t|z|\bigr)\frac{\dd t^p}{\bigl(1+t^{p}\bigr)^2}\in \RR_+
\end{equation}
is finite at each $z\in \CC$ in view of \eqref{icT} and  $M\geq 0$ is subharmonic   \cite[Theorem 2.4.8]{Rans}. Thus, 
\begin{equation}\label{inE}
\int_{\CC}u\dd \mu\leq \int_{\CC} M\dd \mu\quad\text{\it for each Jensen measure $\mu$.}
\end{equation}
Denote by $\mathsf{J}$ the class of Borel measures $\mu$ such that 
\begin{equation}\label{BJ}
    \mathsf{B}_u(1)\overset{\eqref{uB}}{\leq} \int u\dd \mu\quad\text{\it for each subharmonic function $u$ on $\CC$}. 
\end{equation}
If $\mu\in \mathsf{J}$, then $\mu$ is an Jensen measure since 
\begin{equation*}
  u(0)\overset{\eqref{0BC}}{\leq}  \mathsf{B}_u(1)\overset{\eqref{BJ}}{\leq}  \int u\dd \mu\quad\text{\it for each subharmonic function $u$ on $\CC$}.
\end{equation*} 
Therefore, by  \eqref{inE} we obtain  
\begin{equation}\label{inEJ}
\int_{\CC}u\dd \mu\leq \int_{\CC} M\dd \mu\quad\text{\it for each measure $\mu\in \mathsf{J}$.}
\end{equation}

\begin{lemma}[{special case of \cite[Corollary 1.2]{KudMenKha24}}]\label{22} If $u\neq -\infty$ and $M\neq -\infty$ are subharmonic functions such that 
\eqref{inEJ} holds, then   there exists a subharmonic   function $v\neq -\infty$ such that 
\begin{equation}\label{esvp0}
    u(z)+v(z)\leq M(z)\quad\text{for all $z\in \CC$}.
\end{equation}
\end{lemma}
By Lemma \ref{22},  there exists a subharmonic   function $v\neq -\infty$ such that 
\begin{equation}\label{esvp}
    u(z)+v(z)\overset{\eqref{esvp0}}{\underset{z\in \CC}{\leq}} M(z)\overset{\eqref{TI}}{\underset{z\in \CC}{=}}
p\int_0^{+\infty}\mathsf{C}_u\bigl(t|z|\bigr)\frac{\dd t^p}{\bigl(1+t^{p}\bigr)^2}.
\end{equation}

We proceed to the proof of the existence of an entire function $h$ for which   \eqref{esh0} holds.

\begin{lemma}[{\cite[Corollary 2]{KhaBai16}, see also \cite[Lemma 5.1]{BaiKha17}, \cite[Theorem 1]{Kha21e}}]\label{LemmaKhB} For every subharmonic function $v\neq -\infty$ on $\CC$ and for each number $N\in \RR_+$ there is an entire function $h\neq 0$ such that 
\begin{equation}\label{lnhv}
     \ln |h(z)|\underset{z\in \CC}{\leq} \frac{1}{2\pi}\int_0^{2\pi}v\biggl(z+\frac{1}{\bigl(1+|z|\bigr)^N}e^{i\theta}\biggr)\dd \theta.
\end{equation}
\end{lemma}
The integration of inequality \eqref{esvp}  entails inequality
\begin{align*}
\frac{1}{2\pi}\int_0^{2\pi}(u+v)\biggl(z+\frac{1}{\bigl(1+|z|\bigr)^N}e^{i\theta}\biggr)\dd \theta
&\overset{\eqref{esvp}}{\underset{z\in \CC}{\leq}} 
\frac{1}{2\pi}\int_0^{2\pi}M\biggl(z+\frac{1}{\bigl(1+|z|\bigr)^N}e^{i\theta}\biggr)\dd \theta\\
&\overset{\eqref{TI}}{\underset{z\in \CC}{=}}
\frac{1}{2\pi}\int_0^{2\pi}p\int_0^{+\infty}\mathsf{C}_u\Biggl(t\biggl|z+\frac{1}{\bigl(1+|z|\bigr)^N}e^{i\theta}\biggr|\Biggr)\frac{\dd t^p}{(1+t^{p})^2}\dd \theta.
\\
&\underset{z\in \CC}{\leq} p\int_0^{+\infty}\mathsf{C}_u\Biggl(t\biggl(|z|+\frac{1}{\bigl(1+|z|\bigr)^N}\biggr)\Biggr)\frac{\dd t^p}{(1+t^{p})^2},
\end{align*}
Here, on the left-hand side, for the subharmonic function $u$, we have the inequality
\begin{equation*}
u(z)\underset{z\in \CC}{\leq}\frac{1}{2\pi}\int_0^{2\pi}u\biggl(z+\frac{1}{\bigl(1+|z|\bigr)^N}e^{i\theta}\biggr)\dd \theta,
\end{equation*}
and  by Lemma \ref{LemmaKhB}, \textit{for each} $N\in \RR_+$, there exists  an entire function $h\neq 0$ such that \eqref{lnhv} holds. Thus, 
\begin{equation}\label{u+h}
    u(z)+\ln\bigl|h(z)\bigr|\underset{z\in \CC}{\leq} \int_0^{+\infty}\mathsf{C}_u\Biggl(t\biggl(|z|+\frac{1}{\bigl(1+|z|\bigr)^N}\biggr)\Biggr)\frac{p\dd t^p}{(1+t^{p})^2}=:I(z).  
\end{equation}
By replacing the variable
\begin{equation}\label{repz}
    x:=c_zt, \text{ where }c_z\underset{z\in \CC}{:=}|z|+\frac{1}{\bigl(1+|z|\bigr)^N}\underset{z\in \CC}{\geq} |z|,
\end{equation}
we transform the last integral $I(z)$ into the form
\begin{equation}\label{Ie}
    I(z)=\int_0^{+\infty}\mathsf{C}_u(c_zt)\frac{p\dd t^p}{(1+t^{p})^2}
=\int_0^{+\infty}\mathsf{C}_u(x)\frac{c_z^p    p\dd x^p}{(c_z^p+t^{p})^2}
\overset{\eqref{repz}}{\leq}
\int_0^
{+\infty}\mathsf{C}_u(x)\frac{c_z^p    p\dd x^p}{\bigl(|z|^p+x^{p}\bigr)^2} 
\end{equation}
where $\mathsf{C}_u(x)\geq 0$ in view of \eqref{0} and \eqref{0BC}. 
Besides, it follows from the elementary inequality 
\begin{equation*}
    (x+y)^p\leq x^p+p(x+y)^{p-1}y \quad\text{when $x,y\in \RR_+$, $p\geq 1$,}  
\end{equation*} 
that, for $x:=|z|$,  $y:=\dfrac{1}{(1+|z|)^N}$ and $p\geq 1$, 
\begin{equation*}
 c_z^p\overset{\eqref{repz}}{\underset{z\in \CC}{=}}\biggl(|z|+\frac{1}{\bigl(1+|z|\bigr)^N}\biggr)^p\leq 
|z|^p+p   \biggl(|z|+\frac{1}{\bigl(1+|z|\bigr)^N}\biggr)^{p-1} \frac{1}{(1+|z|)^N}
\leq |z|^p+p\bigl(1+|z|\bigr)^{p-1-N}
\end{equation*}
Therefore, we can continue inequality  \eqref{Ie} as
\begin{multline*}
I(z)\leq \int_0^{+\infty}\mathsf{C}_u(x)\frac{|z|^p    p\dd x^p}{(|z|^p+t^{p})^2}
+
p\bigl(1+|z|\bigr)^{p-1-N}
\int_0^{+\infty}\mathsf{C}_u(x)\frac{p\dd x^p}{(|z|^p+x^{p})^2}
\\
=\int_0^{+\infty}\mathsf{C}_u\bigl(t|z|\bigr)\frac{ p\dd t^p}{(1+t^{p})^2}
+\bigl(1+|z|\bigr)^{p-1-N}\int_0^{+\infty}\mathsf{C}_u(x)\frac{p\dd x^p}{(|z|^p+x^{p})^2}.
\end{multline*}
Hence for $|z|\geq 1$, we obtain 
\begin{equation}\label{Icc0}
I(z)\underset{|z|\geq 1}{\leq} 
\int_0^{+\infty}\mathsf{C}_u\bigl(t|z|\bigr)\frac{ p\dd t^p}{(1+t^{p})^2}
+\bigl(1+|z|\bigr)^{p-1-N}\underset{c}{\underbrace{\int_0^{+\infty}\mathsf{C}_u(x)\frac{p\dd x^p}{(1+x^{p})^2}}}
\end{equation} 
where by \eqref{icT}, the last integral is a finite constant $c\in\RR_+$, independent of $N$. 
Due to the arbitrariness in choosing a fixed value of $N$, we can put $N:=p-1$. Then \eqref{Icc0} entails 
\begin{equation*}
I(z)\underset{|z|\geq 1}{\leq} 
\int_0^{+\infty}\mathsf{C}_u\bigl(t|z|\bigr)\frac{ p\dd t^p}{(1+t^{p})^2}+c\quad\text{for all $z\in \CC\setminus \DD$.}
\end{equation*} 
 Thus, according to \eqref{u+h} we obtain 
\begin{equation*}
    u(z)+\ln\bigl|h(z)\bigr|\leq I(z)\leq 
\int_0^{+\infty}\mathsf{C}_u\bigl(t|z|\bigr)\frac{ p\dd t^p}{(1+t^{p})^2}+c\quad\text{for all $z\in \CC\setminus \DD$.}  
\end{equation*}
Hence for a constant $c_1:=\bigl(\mathsf{M}_{u+\ln|h|}(1)\bigr)^+$,  we have 
\begin{equation*}
    u(z)+\ln\bigl|h(z)\bigr|\leq 
\int_0^{+\infty}\mathsf{C}_u\bigl(t|z|\bigr)\frac{ p\dd t^p}{(1+t^{p})^2}+c_1+c\quad\text{for all $z\in \CC$.}  
\end{equation*}
Here, if we consider the entire function $he^{-c-c_1}\neq 0$ instead of the function $h$, retaining the designation $h$ for it, we get the desired inequality  \eqref{esh0}. 
To get rid of the constraint $u(0)\overset{\eqref{0}}{=}0$, consider the proven inequality \eqref{esh0} for the function $u-u(0)$:
\begin{align*}
    u(z)-u(0)+\ln|h(z)|&\underset{z\in \CC}{\leq}p^2 \int_0^{+\infty}
\mathsf{C}_{u-u(0)}\bigl(|z|t\bigr)\frac{t^{p-1}\dd t}{(1+t^p)^2}
\\
&\underset{z\in \CC}{=}p^2 \int_0^{+\infty}\mathsf{C}_{u}\bigl(|z|t\bigr)\frac{t^{p-1}\dd t}{(1+t^p)^2}
-u(0)
p^2 \int_0^{+\infty}\frac{t^{p-1}\dd t}{(1+t^p)^2}\\
&\underset{z\in \CC}{=}p^2 \int_0^{+\infty}\mathsf{C}_{u}\bigl(|z|t\bigr)\frac{t^{p-1}\dd t}{(1+t^p)^2}
-pu(0).
\end{align*}
Here, if we consider the entire function $he^{(p-1)u(0)}\neq0$ instead of the function $h$, retaining the designation $h$ for it, we get the required inequality \eqref{esh0} for any value  $u(0)\in \RR$. 

The proof of Theorem \ref{Th1} is complete. 

\section{Proofs of results for entire and meromorphic functions}

\subsection{Proofs  of main results for entire 
functions}
\begin{proof}[Proof of Theorem\/ {\rm\ref{Cor1}}]
If $f(0)\neq 0$, then the subharmonic function  $u=\ln|f|$ is harmonic 
 in some neighborhood of the origin. By the Main Theorem there is an entire function $h\neq 0$ such that 
\begin{equation*}
    \bigl(\ln|fh|\bigr) (z)\underset{z\in \CC}{=}\ln \bigl|f(z)\bigr|+\ln\bigl|h(z)\bigr|
\overset{\eqref{esh0}}{\underset{z\in \CC}{\leq}} \int_0^{+\infty} \mathsf{C}_{\ln|f|}\bigl( |z|t\bigr)\frac{p^2t^{p-1}\dd t}{(1+t^p)^2}
\quad\text{for all $z\in \CC$}.
\end{equation*}
Hence we obtain 
\begin{equation*}
    \mathsf{M}_{\ln|fh|}(r)=\sup_{|z|=r}\bigl(\ln|fh|\bigr) (z)\underset{r\in \RR_+}{\leq}
\int_0^{+\infty} \mathsf{C}_{\ln|f|}( rt)\frac{p^2t^{p-1}\dd t}{(1+t^p)^2},
\quad\text{for all $r\in \RR_+$}.
\end{equation*}
By  equalities  \eqref{MCTu}, this gives \eqref{eshe}.  

If we have \eqref{M/C}, then by the following  relations
 \begin{equation*}
    \ln M(r;fh)\overset{\eqref{MCTu}}{=}\mathsf{M}_{\ln|fh|}\geq  \mathsf{C}_{\ln|fh|}=\mathsf{C}_{\ln|f|}+\mathsf{C}_{\ln|h|}
=C(r;f)+\mathsf{C}_{\ln|h|},
\end{equation*}
we obtain 
\begin{equation*}
 \liminf_{r\to +\infty} \mathsf{C}_{\ln|h|}\leq \liminf_{r\to +\infty}( \ln M(r;fh)-C(r;f))=
 \ln \liminf_{r\to +\infty} \frac{M(r;fh)}{\exp C(r;f)}\overset{\eqref{M/C}}{=}-\infty. 
\end{equation*}
Hence $\mathsf{C}_{\ln|h|}=-\infty$ and $h=0$ under condition   \eqref{M/C}. 
This completes the proof of Theorem \ref{Cor1}.
\end{proof}

\begin{proof}[Proof of Corollary {\rm \ref{Coruniq}}] 
By condition \eqref{eqCrs}, there is a constant $c\in \RR_+$ such that
\begin{equation*}
C(r;f)\leq \sigma t^{\rho}+c\quad\text{for all $r\in \RR_+$.}
\end{equation*}
For 
$p:= \max\{1,2\rho\}\geq 1$,   
 by Theorem \ref{Cor1},  there is an entire function $h_p\neq 0$ such that 
\begin{equation*}
\ln  M(r;fh_p)
\leq 
\int_0^{+\infty}
\frac{p^2(\sigma (tr)^{\rho}+c)t^{p-1}\dd t}{(1+t^p)^2}
=p\sigma r^{\rho}\int_0^{+\infty}
\frac{ t^{\rho}\dd t^p}{(1+t^p)^2}+pc
\quad\text{for all $r\in \RR_+$},
\end{equation*}
 where the right side is the Euler integral \cite[Ch.~V, \S~2, {\bf 74}, Example 3]{LSh}  with a multiplier:
\begin{equation*}
\int_0^{+\infty}\frac{ t^{\rho}\dd t^p}{(1+t^p)^2}=
\frac{\rho}{p}\int_0^{+\infty}\frac{ x^{\frac{\rho}{p}-1}\dd x}{1+x}
=\frac{\rho}{p}\cdot \frac{\pi}{\sin \frac{\rho}{p} \pi}
\overset{\eqref{Paley}}{=}
\frac{1}{p}P(\rho).
\end{equation*}
Hence we obtain  
\begin{equation}\label{eshers}
\ln  M(r;fh_p) \leq p\sigma r^{\rho}\frac{1}{p}P(\rho)+cp
=P(\rho)\sigma r^{\rho}+cp
\quad\text{for all $r\in \RR_+$},
\end{equation}
and the function $h:=h_pe^{-cp}$ gives  \eqref{Mh}.

In the case of $\rho=1$, the Paley constant $P(1)$ cannot be reduced in  \eqref{Mh} for the function $f\colon z\mapsto 1/z\Gamma(z)$, where $\Gamma$ is the  gamma function. In the case of $0<\rho\neq 1$, the Paley constant $P(1)$ cannot be reduced in  \eqref{Mh} for the function 
$f\colon z \mapsto E_{\rho}(z;1)$, where $E_{\rho}(\cdot;1)$ is the Mittag-Leffler function of order $\rho$ \cite[Ch.~III, \S~2]{Dzh}, \cite{PoSed}. We do not stop here for detailed calculations.
\end{proof}

\subsection{Proofs  of main results on distribution of roots for entire 
functions}

\begin{proof}[Proof of Theorem\/ {\rm\ref{Th2}}] There is an entire function $f_{\mathsf{Z}}$ such that $\Zero_{f_{\mathsf{Z}}}=\mathsf{Z}$ and $f_{\mathsf{Z}}(0)\neq 0$ in view of $\mathsf{Z}(0)=0$. In particular, we have 
\begin{equation*}
N_{\Zero_{f_{\mathsf{Z}}}}(r) =  N_{\mathsf{Z}}(r)\overset{\eqref{NZ}}{=}
\int_0^r\frac{\mathsf{Z}^{\rad}(t)}{t}\dd t=C(r,{f_{\mathsf{Z}}})-\ln \bigl|{f_{\mathsf{Z}}}(0)\bigr|
\quad\text{for all $r\in \RR_+$}.
\end{equation*} 
Then, by Theorem \ref{Cor1}, for every $p\geq 1$  there is an entire function $h\neq 0$ such that 
\begin{equation*}
\ln M(r;f_{\mathsf{Z}}h)
\overset{\eqref{eshe}}{\leq} 
\int_0^{+\infty} C( rt; {f_\mathsf{Z}})\frac{p^2t^{p-1}\dd t}{(1+t^p)^2}
=\int_0^{+\infty}
N_\mathsf{Z}(rt)\frac{p^2t^{p-1}\dd t}{(1+t^p)^2}+p\ln \bigl|{f_{\mathsf{Z}}}(0)\bigr|
\quad\text{for all $r\in \RR_+$},
\end{equation*}
Hence, the entire function $f:=f_{\mathsf{Z}}h\bigl|{f_{\mathsf{Z}}}(0)\bigr|^{-p}$ is 
the desired one.
\end{proof}
We omit the proof of Corollary \ref{CoruniqZ}, since it is constructed similarly to the proof of Corollary \ref{Coruniq}.

\subsection{Proofs  of main results on meromorphic functions}

\begin{proof}[Proof of Theorem\/ {\rm\ref{Th3}}] There is a representation  $F=\dfrac{f_0}{g_0}$ as the ratio of entire functions $f_0$ and $g_0$
without common roots. Then  \cite[Ch. I,\S~4]{GO}
\begin{equation}\label{u_F}
    T(r;F) \underset{r\in \RR_+}{=}\mathsf{C}_{u_F}(r), \quad\text{where $u_F:=\sup\bigl\{\ln|f_0|, \ln|g_0|\bigr\}$}.
\end{equation}
By the Main Theorem, for every $p\geq 1$  there is an entire function $h\neq 0$ such that 
\begin{equation*}
    u_f(z)+\ln|h(z)|\leq p^2 \int_0^{+\infty}
\mathsf{C}_{u_F}\bigl(|z|t\bigr)\frac{t^{p-1}\dd t}{(1+t^p)^2}
\overset{\eqref{u_F}}{=}\int_0^{+\infty}
T\bigl(|z|t;F\bigr)\frac{p^2t^{p-1}\dd t}{(1+t^p)^2}
\quad\text{for all $z\in \CC$.}
\end{equation*}
Hence, by \eqref{u_F},  we obtain   
\begin{equation*}
\sup \{\ln M(r;f_0h),\ln M(r;g_0h)\} \overset{\eqref{u_F}}{=}    \mathsf{M}_{u_F+\ln|h|} (r)
\leq \int_0^{+\infty}
T(rt;F)\frac{p^2t^{p-1}\dd t}{(1+t^p)^2}
\quad\text{for all $r\in \RR_+$,}
\end{equation*}
and $F=\dfrac{f_0h}{g_0h}$. If we put $f=f_0h$ and $g:=g_0h$, then the last inequality is inequality \eqref{eshem}.

Suppose that  there is a pair of entire functions $f,g$ such that $F=\dfrac{f}{g}$  and 
\eqref{M/Cm} holds. Then 
\begin{multline*}
 \liminf_{r\to +\infty}\Bigl(\max\bigl\{\mathsf{C}_{\ln|f|}(r),\mathsf{C}_{\ln|g|}(r)\bigr\}- T(r,F)\Bigr)
\\
\overset{\eqref{0BC}-\eqref{MCTu}}{\leq}     \liminf_{r\to +\infty}\Bigl(\max\bigl\{\mathsf{M}_{\ln|f|}(r),\mathsf{M}_{\ln|g|}(r)\bigr\}- T(r,F)\Bigr)\overset{\eqref{M/Cm}}{=}-\infty.
\end{multline*}
 Hence if we use a representation $F=\dfrac{f_0h}{g_0h}$ where  $f_0,g_0\neq 0$ is a pair of entire functions without common roots and 
the  representation \eqref{u_F} for the Nevalinna characteristic of $F$, then we obtain 
   $ \liminf\limits_{r\to +\infty}\mathsf{C}_{\ln|h|}(r)=-\infty$
and $h=0$. This completes the proof of Theorem \ref{Th3}.
\end{proof} 

We omit the proof of Corollary \ref{CoruniqF}, since it is constructed similarly to the proof of Corollary \ref{Coruniq}.

\end{document}